\documentclass[a4paper,11pt]{article}
\usepackage[latin1]{inputenc}
 \usepackage[T1]{fontenc}
  \usepackage{amsmath,amsthm,amssymb}
 \usepackage[all]{xy}

\def \R{{\mathbb R}}

\def \1{{\mathbb 1}}

\newtheorem{Exemp}{Examples}
\newtheorem{Thm}{Theorem}
\newtheorem{Prop}{Proposition}
\newtheorem{Def}{Definition}
\newtheorem{Rem}{Remark}
\newtheorem{Lem}{Lemma}

\newtheorem{Def Nota}{Definitions and notations} 
\newtheorem{Cor}{Corollary}

\font\ninerm=cmr9
\long\outer\def\abstract#1{\bigskip\vbox{\noindent\ninerm
\baselineskip=10pt#1}\nobreak\bigskip}
\newcount\numero
\def\exo#1{\advance\numero by 1\bigskip
{\noindent\tenbf #1\the\numero. }}
\def\frac#1#2{{#1\over #2}}
\title{On representations of isometric isomorphisms between some monoid of functions.}   
\author{Mohammed Bachir}
\begin{document}
\maketitle
\begin{center} {\it Laboratoire SAMM 4543, Universit\'e Paris 1 Panth\'eon-Sorbonne, Centre P.M.F. 90 rue Tolbiac 75634 Paris cedex 13}
\end{center}
\begin{center} 
{\it Email : Mohammed.Bachir@univ-paris1.fr}
\end{center}
\noindent\textbf{Abstract.} We prove that each isometric isomorphism, between the monoids of all nonegative $1$-Lipschitz maps defined on invariant metric groups and equiped with the inf-convolution law, is given canonically from an isometric isomorphism between their groups of units.  

\vskip5mm

\noindent{\bf Keyword, phrase:} Inf-convolution; $1$-Lipschitz; isometries; group and monoid structure; Banach-Stone theorem.\\
{\bf 2010 Mathematics Subject:} 46T99; 26E99; 20M32; 47B33.\\
\section*{Introduction.}
Given a metric space $(X,d)$, we denote by $Lip^1_+(X)$ the set of all nonnegative $1$-Lipschitz maps on $X$ equipped with the metric
$$\rho(f,g):= \sup_{x\in X} \frac{|f(x)-g(x)|}{1+|f(x)-g(x)|}, \hspace{2mm} \forall f,g\in Lip^1_+(X).$$

If $X$ is a group and $f, g: X\longrightarrow \R$ are two functions, the inf-convolution of $f$ and $g$ is defined by the following formula
\begin{eqnarray*}  \label{Al}
(f\oplus g) (x) &:=& \inf_{y,z\in X/ yz=x}\left\{f(y)+g(z)\right\}; \hspace{3mm}\forall x\in X.
\end{eqnarray*}

We recall the following definition.
\begin{Def} \label{invgroup} 
Let $(X,d)$ be a metric group. We say that $(X,d)$ is invariant if and only if, 
$$d(xy,xz)=d(yx,zx)=d(y,z), \hspace{2mm} \forall x, y, z \in X.$$
If moreover $X$ is complete for the metric $d$, then we say that $(X,d)$ is an invariant complete metric group.
\end{Def}
Examples of invariant metric groups are given in \cite{Ba1}. In all the paper $(X,d)$ and $(Y,d')$ will be assumed to be invariant metric groups having respectively $e$ and $e'$ as identity element and $(\overline{X},\overline{d})$ (resp.  $(\overline{Y},\overline{d'})$) denotes the group completion of $(X,d)$ (resp.  of $(Y,d')$).

\vskip5mm
Recently, we established in \cite{Ba1} that the set $(Lip^1_+(X),\oplus)$ enjoys a monoid structure, having the map $\delta_e: x\mapsto d(x,e)$ as identity element and that the group completion $(\overline{X},\overline{d})$ of $(X,d)$ is completely determined by the metric monoid structure of $(Lip^1_+(X),\oplus, \rho)$. In other words, $(Lip^1_+(X),\oplus, \rho)$ and $(Lip^1_+(Y),\oplus, \rho)$ are isometrically isomorphic as monoids if and only if, $(\overline{X},\overline{d})$ and $(\overline{Y},\overline{d'})$ are isometrically isomorphic as groups. Also, we proved that the group of all invertible elements of $Lip^1_+(X)$ coincides, up to isometric isomorphism, with the group completion $\overline{X}$ (See [Theorem 1., \cite{Ba1}] and [Theorem 2., \cite{Ba1}]). The main result of \cite{Ba1}, gives a Banach-Stone type theorem.

\vskip5mm
The representations of isometries between Banach spaces of Lipschitz maps defined on metric spaces and equipped with their natural norms, was considered by several authors \cite{GJ}, \cite{W}, \cite{JV}. In general, such isometries are given, under some conditions, canonically as a composition operators. Other Banach-Stone type theorems are also given for unital vector lattices structure \cite{GJ1}.
\vskip5mm
In this note, we provide the following result which gives complete representations of isometric isomorphisms for the monoid structure between $Lip^1_+(X)$ and $Lip^1_+(Y)$. Our result complement those given in \cite{Ba1} and \cite{Ba}.

\begin{Thm} \label{Thm0}  Let $(X,d)$ and $(Y,d')$ be two invariant metric groups. Let $\Phi$ be a map from $(Lip^1_+(X),\oplus, \rho)$ into $(Lip^1_+(Y),\oplus, \rho)$. Then the following assertions are equivalent.

$(1)$ $\Phi$  is an isometric isomorphism of monoids

$(2)$ there exists an isometric isomorphism of groups $T: (\overline{X},\overline{d})\longrightarrow (\overline{Y},\overline{d'})$ such that $\Phi (f)=(\overline{f}\circ T^{-1})_{|Y}$ for all 
$f\in Lip^1_+(X)$, where $\overline{f}$ denotes the unique $1$-Lipschitz extenstion of $f$ to $\overline{X}$ and $(\overline{f}\circ T^{-1})_{|Y}$ denotes the restriction of $\overline{f}\circ T^{-1}$ to $Y$.
\end{Thm}

\vskip5mm

If $A$ (resp. X) is a metric monoid (resp. a metric group), by $Is_m(A)$ (resp. $Is_g(X)$) we denote the group of all isometric automorphism of the monoid $A$ (resp. of the group $X$). 
The symbol ''$\simeq$'' means "{\it isomorphic as groups}". An immediate consequence of Theorem \ref{Thm0} is given in the following corollary.

\begin{Cor} \label{Corintro} Let $(X,d)$ be an invariant metric group. Then,
$$Is_m(Lip^1_+(X)) \simeq Is_g(\overline{X}).$$
\end{Cor}
As application of the results of this note, we discover new semigroups law on $\R^n$ (different from the usual operation $+$) having some nice properties. We treat this question in Example \ref{Exemp1} at Section \ref{S4}, where it is shown that each finite group structure $(G, \cdot)$, extend canonically to a semigroup structure on $\R^{n}$ (where $n$ is the cardinal of $G$). In other words, there always exists a semigroup law $\star_G$ on $\R^{n}$ and an injective group morphism $i$ from $(G, \cdot)$ into $(\R^{n},\star_G)$ such that the maximal subgroup of $(\R^{n},\star_G)$ having $e:=(0,1,1,...,1)$ as identity element is isomorphic to the group $G\times \R$. The idea is simply based on the use of the results of this paper and the identification between $(\R^{n},\star_G)$ and $(Lip(G),\oplus)$ where $G$ is equiped with the discrete metric, and $Lip(G)$ denotes the space of all Lipschitz map on $G$.

\vskip5mm
This note is organized as follows. Section \ref{S0} concern the proof of Theorem \ref{Thm0} and is divided on two subsections: in Subsection \ref{S1} we prove some useful lemmas and in Subsection \ref{S2}, we give the proof of the main result Theorem \ref{Thm0}. In Section \ref{S3}, we give some properties of the group of invertible elements for the inf-convolution law. In section \ref{S4}, we review the results of this paper in the algebraic case. 

\section{Proof of the main result.} \label{S0}
\subsection{Preliminary results} \label{S1}
We follow the notation of \cite{Ba1}. For each fixed point $x\in X$, the map $\delta_{x}$ is defined from $X$ into $\R$ as follows
\begin{eqnarray} \delta_{x}: X &\rightarrow& \R \nonumber\\ 
                                         z &\mapsto&  d(z,x)=d(zx^{-1},e).\nonumber
\end{eqnarray} 

We define the subset $\mathcal{G}(X)$  of $Lip^1_+(X)$ as follows 
$$\mathcal{G}(X):=\left\{\delta_{x}: x\in X\right\}\subset Lip^1_+(X).$$

We consider the operator $\gamma_X$ defined as follows
\begin{eqnarray} \gamma_X : X & \rightarrow & \mathcal{G}(X) \nonumber \\
x & \mapsto & \delta_{x} \nonumber 
\end{eqnarray}

We are going to prove some lemmas.
\begin{Lem} \label{lem3} Let $(X,d)$ and $(Y,d')$ be two invariant complete metric groups having respectively $e$ and $e'$ as identity elements. Let $\Phi$ be a map from $(Lip^1_+(X),\oplus, \rho)$ onto $(Lip^1_+(Y),\oplus, \rho)$ which is an isometric isomorphism of monoids. Then, the following asserions holds.

$(1)$ for all $f\in Lip^1_+(X)$, $\inf_{Y} \Phi(f)=\inf_{X} f$  and for all $r\in \R^+$, $\Phi(r)=r$.

$(2)$ there exists an isometric isomorphism of groups $T : (X,d)\longrightarrow (Y,d')$ such that $\Phi(r+\delta_x)=r+\delta_{T(x)}=r+ \delta_x \circ T^{-1}$, for all $r\in \R^+$ and for all $x\in X$.

$(3)$ $\Phi(f+r)=\Phi(f)+ r$, for all $f\in Lip^1_+(X)$ and for all $r\in \R^+$.
\end{Lem}

\begin{proof}

Since an isomorphism of monoids, sends the group of unit onto the group of unit, then using [Theorem 1., \cite{Ba1}], the restriction $T_1:=\Phi_{|\mathcal{G}(X)}$ is an isometric group isomorphism from $\mathcal{G}(X)$ onto $\mathcal{G}(Y)$. On the other hand, the map $\gamma_X : X\longrightarrow \mathcal{G}(X)$ gives an isometric group isomorphism by [Lemma 2., \cite{Ba1}]. Thus, the map $T:=\gamma^{-1}_Y \circ T_1 \circ \gamma_X$, gives an isometric group isomorphism from $X$ onto $Y$ and we have that for all $x\in X$, $\Phi(\delta_x):=T_1(\delta_x)=T_1 \circ \gamma_X(x)=\gamma_Y\circ T(x)=\delta_{T(x)}=\delta_x \circ T^{-1}$.

\vskip5mm 
We prove the part $(1)$. 
Note that $f\oplus 0=0\oplus f=\inf_{x\in X} f$ for all $f\in Lip^1_+(X)$. First, we prove that $\Phi(0)=0$. Indeed, for all $x\in X$, we have that $0\oplus \delta_x=0$. Thus, $\Phi(0)=\Phi(0)\oplus \Phi(\delta_x)=\Phi(0)\oplus\delta_{Tx}$. Using the surjectivity of $T$, we obtain that for all $y\in Y$, we have that $\Phi(0)=\Phi(0)\oplus \delta_y$. So, using the definition of the if-convolution, we get $\Phi(0)(z)=\inf_{ts=z}\lbrace \Phi(0)(t)+\delta_{y}(s) \rbrace \leq \Phi(0)(zy^{-1})$ for all $y,z\in Y$. By taking the infinimum over $y\in Y$, we obtain that $\Phi(0)(z)\leq \inf_Y \Phi(0)$, for all $z\in Y$. It follows that $\Phi(0)=\inf_Y \Phi(0)$ is a constant function. Now, since $\Phi(0)$ is a constant function, we have $2\Phi(0)= \Phi(0)\oplus \Phi(0)=\Phi(0\oplus 0)=\Phi(0)$, it follows that $\Phi(0)=0$. Finaly, we prove that $\Phi(r)=r$  for all $r\in \R^+$. Indeed, since $r\oplus 0=r$ and $\Phi(0)=0$, it follows that $\Phi(r)=\Phi(r)\oplus 0=\inf_Y \Phi(r)$, which implies that $\Phi(r)$ is a constant function. Using the fat that $\Phi$ is an isometry, we get that $\rho(\Phi(r),0)=\rho(\Phi(r),\Phi(0))=\rho(r,0)$. In other word, $\frac{\Phi(r)}{1+\Phi(r)} =\frac{r}{1+r}$, which implies that $\Phi(r)=r$. Now, we have $\inf_{y\in Y} \Phi(f)=\Phi(f)\oplus 0=\Phi(f)\oplus \Phi(0)=\Phi(f\oplus 0)=\Phi(\inf_{x\in X} f)=\inf_{x\in X} f$.

\vskip5mm
We prove the part $(2)$. 
Let $r\in \R^+$ and set $g=\Phi(r+\delta_e)\in Lip^1_+(Y)$. We first prove that $g=r+\delta_{e'}$. Using the part $(1)$, we have that $r=\Phi(r)=\Phi(\inf_{x\in X}(r+ \delta_e))=\inf_{y\in Y} \Phi(r+\delta_e)\leq \Phi(r+\delta_e)=g$. Thus $g-r\geq 0$ and so $g-r \in Lip^1_+(Y)$. On the other hand, since $Lip^1_+(Y)$ is a monoid having $\delta_{e'}$ as identity element, we have that $g=(g-r)\oplus(r+\delta_{e'})=(r+\delta_{e'})\oplus (g-r)$. Now, since $\Phi^{-1}$ is a monoid morphism, we get that 
\begin{eqnarray*}
r+\delta_e &=& \Phi^{-1}(g)\\
           &=&\Phi^{-1}(g-r)\oplus \Phi^{-1}(r+\delta_{e'})=\Phi^{-1}(r+\delta_{e'})\oplus \Phi^{-1}(g-r).
\end{eqnarray*} 
As above we prove that $\Phi^{-1}(r+\delta_{e'})-r\geq 0$. Thus, $\Phi^{-1}(r+\delta_{e'})-r\in Lip^1_+(X)$. Since $r$ is a constant function, the above equality is equivalent to the following one 
\begin{eqnarray*}
\delta_e &=&\Phi^{-1}(g-r)\oplus (\Phi^{-1}(r+\delta_{e'})-r)=(\Phi^{-1}(r+\delta_{e'})-r)\oplus \Phi^{-1}(g-r).
\end{eqnarray*}
Since from [Theorem 1, \cite{Ba1}], the invertible element in $Lip^1_+(X)$ are exactely the element of $\mathcal{G}(X)$ and since $\mathcal{G}(X)$ is a group by [Lemma 2, \cite{Ba1}], we deduce from the above equality that $\Phi^{-1}(r+\delta_{e'})-r \in \mathcal{G}(X)$ and $\Phi^{-1}(g-r) \in \mathcal{G}(X)$ and there exists $\alpha(r), \beta(r) \in X$ such that  
\begin{equation*}
 \left\{
\begin{array}{cl}
&e = \alpha(r)\beta(r)\\
&\Phi^{-1}(r+\delta_{e'})-r = \delta_{\alpha(r)}\\
&\Phi^{-1}(g-r) = \delta_{\beta(r)}
\end{array}
\right.
\end{equation*}
This implies that 
\begin{equation} \label{eq0}
 \left\{
\begin{array}{cl}
&e = \alpha(r)\beta(r)\\
&\Phi(r+\delta_{\alpha(r)})=r+ \delta_{e'}\\   
& g= r+\Phi(\delta_{\beta(r)}) =r+\delta_{T(\beta(r))}
\end{array}
\right.
\end{equation}
We need to prove that $\alpha(r)=\beta(r)=e$ for all $r\in \R^+$. Indeed, since $\Phi$ is an isometry, we have that 
\begin{eqnarray*}
\rho(\Phi(r+\delta_{\alpha(r)}),\Phi(\delta_e))=\rho(r+\delta_{\alpha(r)},\delta_e).
\end{eqnarray*}
Using the above formula, the second equations in (\ref{eq0}) and the definition of the metric $\rho$ with the fact that $\Phi(\delta_e)=\delta_{e'}$, we get

\begin{eqnarray*}
\frac{r}{1+r}&=&\rho(r+ \delta_{e'},\delta_{e'})\\
             &=&\rho(\Phi(r+\delta_{\alpha(r)}),\Phi(\delta_e))\\ 
             &=&\rho(r+\delta_{\alpha(r)},\delta_e)\\
             &=&\sup_{t\in X}\frac{|r+\delta_{\alpha(r)}(t)-\delta_e(t)|}{1+|r+\delta_{\alpha(r)}(t)-\delta_e(t)|}.\\
             &\geq& \frac{r+\delta_{\alpha(r)}(e)}{1+r+\delta_{\alpha(r)}(e)}
\end{eqnarray*}
A simple computation of the above inequality, gives that $\delta_{\alpha(r)}(e)\leq 0$ i.e. $d(\alpha(r),e)\leq 0$. In other word, we have that $\alpha(r)=e$ for all $r\in \R^{+}$. On the other hand, using the first equation of (\ref{eq0}), we get that $\beta(r)=e$ for all $r\in \R^{+}$. It follows from the equation (\ref{eq0}) that $\Phi(r+\delta_e)=r+\delta_{e'}$ for all $r\in \R^+$. Now, it is easy to see that for all $r\in R^+$ and all $x\in X$ we have  
\begin{eqnarray*}
r+\delta_x=(r+\delta_e)\oplus \delta_x.
\end{eqnarray*}
It follows that 
\begin{eqnarray*}
\Phi(r+\delta_x)&=&\Phi(r+\delta_e)\oplus \Phi(\delta_x)\\
                &=& (r+\delta_{e'})\oplus \delta_{T(x)}\\
                &=& r+\delta_{T(x)}
\end{eqnarray*}
Since $T$ is isometric, we obtain that $\Phi(r+\delta_x)=r+\delta_{T(x)}=r+\delta_{x}\circ T^{-1}$. 

\vskip5mm
Now, we prove the part $(3)$. Let $f \in Lip^1_+(X)$ and $r\in \R^+$. It is easy to see that $f+r=f\oplus (r+\delta_e)$. So, using the part $(2)$, we obtain that $\Phi(f+r)=\Phi(f)\oplus \Phi(r+\delta_e)=\Phi(f)\oplus (r+\delta_{e'})=\Phi(f)+r$.
\vskip5mm

\end{proof}

\begin{Lem} \label{lem1.2} Let $(X,d)$ be an invariant metric group. Let $f\in Lip^1_+(X)$. Then, for all $x \in X$ and all positive real number $a$ such that $a\geq f(x)$, we have that
$$f(x)= (\inf(\delta_e,a)\oplus f)(x).$$
\end{Lem}
\begin{proof} Let $x\in X$ and $a \geq 0$ such that $f(x)\leq a$. We have that 
\begin{eqnarray*}
(\inf(\delta_e,a) \oplus f)(x) = \inf_{t \in X} \lbrace \inf (d(xt^{-1},e),a) + f(t) \rbrace\\            
                               = \inf_{t \in X} \lbrace f(t) + \inf (d(t,x),a)\rbrace\\
                               = \min \lbrace \inf_{t\in X/ d(t,x)\leq a} \lbrace f(t) + \inf (d(t,x),a)\rbrace ; \\
                                \inf_{t\in X/ d(t,x)\geq a} \lbrace f(t) + \inf (d(t,x), a)\rbrace \rbrace\\
                               = \min  \lbrace \inf_{t/ d(t,x)\leq a} \lbrace f(t) + d(t,x)\rbrace, \inf_{t/ d(t,x)\geq a} \lbrace f(t) + a\rbrace \rbrace.
\end{eqnarray*}

Since $f$ is $1$-Lipschitz we have that $ f(x)=\inf_{t/ d(t,x)\leq a} \lbrace f(t) + d(t,x)\rbrace$. It follows that 
                              
\begin{eqnarray*}                              
(\inf(\delta_e,a) \oplus f)(x) &=& \min \lbrace f(x), \inf_{t/ d(t,x)\geq a} \lbrace f(t)\rbrace +a \rbrace\\
                               &=& f(x).
\end{eqnarray*}

\end{proof}
\begin{Lem} \label{lem1.3} Let $(X,d)$ be an invariant metric group. Then, the following assertions hold. 

$(1)$ for each $f\in Lip^1_+(X)$ and for each bounded function $h\in Lip^1_+(X)$, the function $f\oplus h \in Lip^1_+(X)$ is bounded.

$(2)$ Let $f, g \in Lip^1_+(X)$, then the following assertions are equivalent. 

\hspace{4mm} $(a)$ $f\leq g$

\hspace{4mm} $(b)$ $h\oplus f \leq h \oplus g$, for all function $h \in Lip^1_+(X)$ which is bounded. 
\end{Lem}

\begin{proof} $(1)$ Since $0 \leq f\oplus h (x)\leq f(e)+h(x)$ for all $x\in X$ and since $h$ is bounded, it follows that $f\oplus h$ is bounded. On the other hand, $f\oplus h \in Lip^1_+(X)$ since $Lip^1_+(X)$ is a monoide.

$(2)$ The part $(a) \Longrightarrow (b)$ is easy. Let us prove the part $(b) \Longrightarrow (a)$. Indeed, let $x\in X$ and chose a positive real number $a\geq \max(f(x), g(x))$. Set $h:= \inf(\delta_e,a)$. It is clear that $h \in Lip^1_+(X)$ and is bounded. So, from the hypothesis $(b)$ we have that $(\inf(\delta_e,a)\oplus f) \leq (\inf(\delta_e,a)\oplus g)$. Using Lemme \ref{lem1.2}, we obtain that $f(x) \leq g (x)$. 

\end{proof}

\begin{Lem} \label{lem1} Let $A$ be a nonempty set and $f,g :A\longrightarrow \R$ be two functions. Then, the following assertions are equivalent.

$(1)$ $\sup_{x\in A}|f(x)-g(x)| <+\infty$.

$(2)$ $\sup_{x\in A} \frac{|f(x)-g(x)|}{1+|f(x)-g(x)|}< 1$.
\end{Lem}
\begin{proof} Suppose that $(1)$ hold. Using [Lemma 1., \cite{Ba1}], we have that $\sup_{x\in A} \frac{|f(x)-g(x)|}{1+|f(x)-g(x)|}=\frac{\sup_{x\in A}|f(x)-g(x)|}{1+\sup_{x\in A}|f(x)-g(x)|} < 1$. Now, suppose that $(2)$ holds. Set $\alpha=\sup_{x\in A} \frac{|f(x)-g(x)|}{1+|f(x)-g(x)|}<1$. Then, we obtain that $|f(x)-g(x)| \leq\frac{\alpha}{1-\alpha}$, for all $x\in A$. This implies that $\sup_{x\in A}|f(x)-g(x)| <+\infty$.

\end{proof}
\begin{Lem} \label{lem2} Let $(X,d)$ and $(Y,d')$ be two invariant complete metric groups. Let 
$$\Phi: (Lip^1_+(X),\rho)\longrightarrow (Lip^1_+(Y), \rho)$$
be an isometric isomorphism of monoids. Then, for all $f, g \in Lip^1_+(X)$, we have 
$$f\leq g \Longleftrightarrow \Phi(f) \leq \Phi(g).$$

\end{Lem}
\begin{proof} The proof is divided on two cases.

{\it Case1: (The case where $f$ and $g$ are bounded.)} Let $f,g \in Lip^1_+(X)$ be bounded functions. In this case we have $\sup_{x\in X}|f(x)-g(x)|< +\infty$, so using Lemma \ref{lem1} and the fact that $\Phi$ is isometric, we get also that $\sup_{y\in Y}|\Phi(f)(y)-\Phi(g)(y)|< +\infty$. Using [Lemma 1. \cite{Ba1}] and the fact that $\Phi$ is isometric, we obtain that
$$\frac{\sup_{y\in Y}|\Phi(f)(y)-\Phi(g)(y)|}{1+\sup_{y\in Y}|\Phi(f)(y)-\Phi(g)(y)|}=\frac{\sup_{x\in X}|f(x)-g(x)|}{1+\sup_{x\in X}|f(x)-g(x)|}.$$
This implies that
$$\sup_{y\in Y}|\Phi(f)(y)-\Phi(g)(y)|=\sup_{x\in X}|f(x)-g(x)|.$$
Set $r:=\sup_{y\in Y}|\Phi(f)(y)-\Phi(g)(y)|=\sup_{x\in X}|f(x)-g(x)|<+\infty$. By applying the above arguments to $f+r$ and $g$ which are bounded, we also get that 
$$\sup_{y\in Y}|\Phi(f+r)(y)-\Phi(g)(y)|=\sup_{x\in X}|(f+r)(x)-g(x)|.$$ 
Using the fact that $\Phi(f+r)=\Phi(f)+r$ (by Lemma \ref{lem3}) and the choice of the number $r$, we get that
$$\sup_{x\in X}\lbrace \Phi(f)(x)-\Phi(g)(x) +r \rbrace=\sup_{x\in X}\lbrace f(x)-g(x)+r \rbrace$$
which implies that 
$$\sup_{y\in Y}\lbrace \Phi(f)(y)-\Phi(g)(y) \rbrace=\sup_{x\in X}\lbrace f(x)-g(x)\rbrace.$$
It follows that $f\leq g \Longleftrightarrow \Phi(f) \leq \Phi(g).$ Replacing $\Phi$ by $\Phi^{-1}$ we also have $k\leq l \Longleftrightarrow \Phi^{-1}(k) \leq \Phi^{-1}(l),$ for all bounded functions $k, l \in Lip^1_+(Y)$.

{\it Case2: (The general case.)} First, note that for each bounded function $k\in Lip^1_+(Y)$, we have that $\Phi^{-1}(k) \in Lip^1_+(X)$ is bounded. Indeed, there exists $r\in \R^+$ such that $0\leq k \leq r$. Using the above case, we get that $\Phi^{-1}(0) \leq \Phi^{-1}(k)\leq \Phi^{-1}(r)$. This shows that $\Phi^{-1}(k)$ is bounded, since $\Phi^{-1}(0)=0$ and $\Phi^{-1}(r)=r$ by Lemma \ref{lem3}. 

Now, let $f,g \in Lip^1_+(X)$ be two functions such that $f\leq g$. Let $k\in Lip^1_+(Y)$ be any bounded function. It follows that $f \oplus \Phi^{-1}(k) \leq g \oplus \Phi^{-1}(k)$. From the part $(1)$ of Lemma \ref{lem1.3}, we have that $f \oplus \Phi^{-1}(k), g \oplus \Phi^{-1}(k) \in Lip^1_+(X)$ are bounded. Using {\it Case1.}, we get that $\Phi(f \oplus \Phi^{-1}(k))\leq \Phi(g \oplus \Phi^{-1}(k))$. Since $\Phi$ is a morphism, we have that $\Phi(f) \oplus k \leq \Phi(g) \oplus k$, which implies that $\Phi(f) \leq \Phi(g)$ by using the part $(2)$ of Lemma \ref{lem1.3}. The converse is true by changing $\Phi$ by $\Phi^{-1}$.


\end{proof}

\begin{Lem} \label{lem4} Let $(X,d)$ and $(Y,d')$ be two invariant metric groups and let $\Phi$ be a monoid isomorphism $\Phi: (Lip^1_+(X),\oplus, \rho)\longrightarrow (Lip^1_+(Y),\oplus, \rho)$. Then, the following assertions are equivalent.

$(1)$ for all $f, g \in Lip^1_+(X)$, we have that $(f\leq g \Longleftrightarrow \Phi(f)\leq \Phi(g)).$

$(2)$ for all $f', g' \in Lip^1_+(Y)$, we have that $(f'\leq g' \Longleftrightarrow \Phi^{-1}(f')\leq \Phi^{-1}(g')).$
 
$(3)$ for all familly $(f_i)_{i\in I}\subset Lip^1_+(X)$, where $I$ is any nonempty set, we have 
$\Phi(\inf_{i\in I} f_i)=\inf_{i\in I} \Phi(f_i)$.
\end{Lem}

\begin{proof} The part $(1)\Longleftrightarrow (2)$ is clear. Let us prove $(1)\Longrightarrow (3)$. Let $(f_i)_{i\in I}\subset Lip^1_+(X)$, where $I$ is any nonempty set. First, it is easy to see that the infinimum of a nonempty familly of nonnegative and $1$-Lipschitz functions is also nonnegative and $1$-Lipschitz function. So, $\inf_{i\in I} f_i \in Lip^1_+(X)$. For all $i\in I$, we have that $\inf_{i\in I} f_i \leq f_i$, which implies by hypothesis that $\Phi(\inf_{i\in I} f_i)\leq \Phi(f_i)$ for all $i\in I$. Consequently we have that $\Phi(\inf_{i\in I} f_i)\leq \inf_{i\in I}\Phi(f_i)$. On the other hand, since $\inf_{i\in I}\Phi(f_i)\leq \Phi(f_i)$ for all $i\in I$, using $(2)$, we have that $\Phi^{-1}(\inf_{i\in I}\Phi(f_i))\leq f_i$, for all $i\in I$. It follows that, $\Phi^{-1}(\inf_{i\in I}\Phi(f_i))\leq \inf_{i\in I} f_i$. Using $(1)$, we obtain that $\inf_{i\in I}\Phi(f_i)\leq \Phi(\inf_{i\in I} f_i)$. Hence, $\inf_{i\in I}\Phi(f_i)= \Phi(\inf_{i\in I} f_i)$. Now, let us prove that $(3)\Longrightarrow (1)$. First, let us show that from $(3)$ we also have that $\Phi^{-1}(\inf_{i\in I} g_i)=\inf_{i\in I} \Phi^{-1}(g_i)$, where $I$ is a nonempty set and $g_i\in Lip^1_+(Y)$ for all $i\in I$. Indeed, since $\Phi$ is bijective, there exists $(f_i)_{i\in I}\subset Lip^1_+(X)$ such that $g_i=\Phi(f_i)$ for all $i\in I$. Thus, $\inf_{i\in I} g_i=\inf_{i\in I} \Phi(f_i)=\Phi(\inf_{i\in I} f_i)=\Phi(\inf_{i\in I} \Phi^{-1}(g_i))$, which implies that $\Phi^{-1}(\inf_{i\in I} g_i)=\inf_{i\in I} \Phi^{-1}(g_i)$. Now, let $f, g \in Lip^1_+(X)$. We have that $f\leq g \Longleftrightarrow f=\inf(f,g)$, so if $f\leq g$ then $\Phi(f)=\Phi(\inf(f,g))=\inf(\Phi(f),\Phi(g))$. This implies that $\Phi(f)\leq \Phi(g)$. Conversely, if $\Phi(f)\leq \Phi(g)$ then $\Phi(f)=\inf(\Phi(f), \Phi(g))$ and so $f=\Phi^{-1}(\Phi(f))=\Phi^{-1}(\inf(\Phi(f), \Phi(g)))=\inf(\Phi^{-1}(\Phi(f)),\Phi^{-1}(\Phi(g)))=\inf(f,g)$. This implies that $f\leq g$.

\end{proof}
\subsection{Proof of the main result.} \label{S2}
Now, we give the proof of the main result.

\begin{proof}[Proof of Theorem \ref{Thm0}.] We know from [Lemma 3. , \cite{Ba1}] that the map 
\begin{eqnarray} \chi_X : (Lip^1_+(X),\oplus,\rho) &\rightarrow& (Lip^1_+(\overline{X}),\oplus,\rho)\nonumber\\
                            f &\mapsto& \overline{f}\nonumber
 \end{eqnarray}
is an isometric isomorphism of monoids, where $\overline{f}$ denotes the unique $1$-Lipschitz extention of $f$ to $\overline{X}$. 
Let us define the map $\overline{\Phi} : (Lip^1_+(\overline{X}),\oplus,\rho) \longrightarrow (Lip^1_+(\overline{Y}),\oplus,\rho)$ by $\overline{\Phi}:= \chi_Y \circ \Phi \circ \chi^{-1}_X$. Then, $\overline{\Phi}$ is an isometric isomorphism of monoids. 

$(1) \Longrightarrow (2)$. Since $Lip^{1}_+(\overline{X})$ is a monoid having $\delta_e: \overline{X}\ni x\mapsto \overline{d}(x,e)$ as identity element, we have that $\overline{f}=\delta_e\oplus \overline{f}$ for all $\overline{f}\in Lip^{1}_+(\overline{X})$. Thus, $\overline{f}=\inf_{t\in \overline{X}}\lbrace \overline{f}(t)+\delta_t\rbrace$ for all $\overline{f}\in Lip^{1}_+(\overline{X})$. Using Lemma \ref{lem4} together with Lemma \ref{lem2}, we have that for all $\overline{f}\in Lip^1_+(\overline{X})$, $\overline{\Phi}(\overline{f})=\overline{\Phi}(\inf_{t\in \overline{X}}\lbrace \overline{f}(t)+\delta_t\rbrace)=\inf_{t\in \overline{X}} \overline{\Phi}(\overline{f}(t)+\delta_t)$. Using Lemma \ref{lem3}, there exists an isometric isomorphism of groups $T:(\overline{X},\overline{d}) \longrightarrow (\overline{Y},\overline{d'})$ such that $\overline{\Phi}(\overline{f}(t)+\delta_t)=\overline{f}(t)+ \delta_{T(t)}$, for all $t\in \overline{X}$. Thus, we get that $\overline{\Phi}(\overline{f})=\inf_{t\in \overline{X}}\lbrace \overline{f}(t)+ \delta_{T(t)}\rbrace$. Equivalently, for all $y\in \overline{Y}$, we have 
\begin{eqnarray*}
\overline{\Phi}(\overline{f})(y)&=& \inf_{t\in \overline{X}}\lbrace \overline{f}(t)+ \delta_{T(t)}(y)\rbrace\\
          &=& \inf_{t\in \overline{X}}\lbrace \overline{f}(t)+ \overline{d'}(y,T(t))\rbrace\\ 
          &=& \inf_{t\in \overline{X}}\lbrace \overline{f}(t)+ \overline{d}(T^{-1}(y),t)\rbrace\\
          &=& (\delta_e\oplus \overline{f})(T^{-1}(y))\\
          &=& \overline{f}(T^{-1}(y))\\
          &=& \overline{f}\circ T^{-1} (y).
\end{eqnarray*}
From the formulas $\Phi=\chi^{-1}_Y\circ \overline{\Phi}\circ \chi_X$, we get that $\Phi (f)=(\overline{f}\circ T^{-1})_{|Y}$ for all 
$f\in Lip^1_+(X)$. 

$(2) \Longrightarrow (1)$. If $T:(\overline{X},\overline{d}) \longrightarrow (\overline{Y},\overline{d'})$ is an isometric isomorphism of groups, then clearly the map $\overline{\Phi}$ defined by $\overline{\Phi}(\overline{f}):=\overline{f}\circ T^{-1}$ for all $\overline{f}\in Lip^1_+(\overline{X})$, gives an isometric isomorphism from $(Lip^1_+(\overline{X}),\oplus,\rho)$ onto $(Lip^1_+(\overline{Y}),\oplus,\rho)$. Thus, the map $\Phi:=\chi^{-1}_Y\circ \overline{\Phi}\circ \chi_X$ gives an isometric isomorphism from $(Lip^1_+(X),\oplus,\rho)$ onto $(Lip^1_+(Y),\oplus,\rho)$. Now, it clear that $\Phi(f)=(\overline{f}\circ T^{-1})_{|Y}$ for all $f\in Lip^1_+(X)$.

\end{proof}

\begin{Rem} \label{rem1} $(1)$ The description of all isomorphisms seems to be more complicated than the representations of the isometric isomorphisms. Here is two examples of isomorphisms which are not isometric.

\hspace{3mm} $(a)$ The map $\Phi : Lip^1_+(X) \longrightarrow Lip^1_+(X)$ defined by $\Phi(f)=f+\inf_X(f)$ for all $f \in Lip^1_+(X)$, is an isomorphism of monoids which respect the order but is not isometric for $\rho$ (the proof is similar to the proof of [Theorem 7., \cite{Ba}]. Note that we always have $\inf_{X}(f\oplus g)=\inf_X(f)+\inf_Y(g)$). 

\hspace{3mm} $(b)$ The map $\Phi : Lip^1_+(\R) \longrightarrow Lip^1_+(\R)$ defined by $\Phi(f)(x)=f(x+\inf_X(f))$ for all $f \in Lip^1_+(\R)$ and all $x\in \R$, is an isomorphism but not isometric for $\rho$.

$(2)$ Following the proof of Theorem \ref{Thm0} and changing "$1$-Lipschitz function" by "$1$-Lipschitz and convex function", we get a positive answer to the problem 2. in \cite{Ba}.
\end{Rem}

\section{The group of units.} \label{S3}

In order that the inf-convolution of two functions $f$ and $g$ takes finit values i.e $f\oplus g >-\infty$, we need to assume that $f$ and $g$ are bound from below. Since, we work with Lipschitz maps, for simplicity, we assume in this section, that $(X,d)$ is a bounded invariant metric group. By $Lip^1_0(X)$ we denote the set of all $1$-Lipschitz map $f$ from $X$ into $\R$ such that $\inf_X(f)=0$. By $Lip^1(X)$ (resp. $Lip(X)$, ) we denote the set of all $1$-Lipschitz map (resp. the set of all Lipchitz map) defined from $X$ to $\R$. We have that 
$$Lip^1_0(X) \subset Lip^1_+(X) \subset Lip^1(X) \subset Lip(X).$$ 

\begin{Prop} \label{MS} Let $(X,d)$ be a bounded invariant metric (abelian) group. Then, the sets $Lip^1_0(X)$, $Lip^1_+(X)$ and $Lip^1(X)$ are (abelian) monoids having $\delta_e$ as identity element and $Lip(X)$ is a (abelian) semigroup.
\end{Prop}
\begin{proof} The proof is similar to [Proposition 1., \cite{Ba1}].
\end{proof}

Note that since $(X,d)$ is bounded, each function $f \in Lip^1(X)$ (resp. $f \in Lip(X)$) is dounded and so $d_{\infty}(f,g):=\sup_{x\in X}|f(x)-g(x)|< +\infty$ for all $f, g \in Lip^1(X)$ (resp. $f,g \in Lip(X)$). In this case, from [Lemma 1., \cite{Ba1}], we have that $$\rho=\frac{d_{\infty}}{1+d_{\infty}}$$ on $Lip(X)$. We also consider the following metric: 
$$\theta_{\infty}(f, g):= d_{\infty}(f-\inf_X(f),g-\inf_X(g))+|\inf_X(f)-\inf_X(g)|, \hspace{3mm} \forall f, g \in Lip(X).$$ 

\begin{Prop} \label{Ptau} Let $(X,d)$ be a bounded invariant metric group. Then, the following map
\begin{eqnarray*}
\tau : (Lip^1(X),\theta_{\infty}) &\longrightarrow& (Lip^1_0(X)\times \R,d_{\infty}+|.|)\\
                          f   &\mapsto & (f-\inf_X(f),\inf_X(f)).
\end{eqnarray*}
is an isomeric isomorphism of monoids, where $Lip^1_+(X)\times \R$ is equiped with the operation $\overline{\oplus}$ defined by $(f,c)\overline{\oplus} (f',c'):=(f\oplus f',c+c')$.. 
\end{Prop}
\begin{proof} Clearly, $(Lip^1_+(X)\times \R,\overline{\oplus})$ is a monoid having $(\delta_e,0)$ as identity element, since $(Lip^1_+(X),\oplus)$ is a monoid having $\delta_e$ as identity element. It is also clear that $\tau$ is a monoid isomorphism. Now, $\tau$ is isometric by the defintion of $\theta_{\infty}$. It follows that $\tau$ is an isometric isomorphism,

\end{proof}
The following proposition gives an alternative way to consider the group completion of invariant metric groups. Recall that if $(M,\cdot)$ is a monoid having $e_M$ as identity element, the group of units of $M$ is the set $$\mathcal{U}(M):=\lbrace m\in M/\hspace{1mm} \exists m'\in M: m\cdot m'=m'\cdot m=e_M \rbrace.$$
The symbol $\cong$ means isometrically isomorphic as groups.  We give below an analogue to [Corollary 1., \cite{Ba1}], for each of the spaces $Lip^1_0(X), Lip^1(X)$ and $Lip(X).$ Note that in the part $(1)$ of the following proposition as in [Corollary 1.,\cite{Ba1}], we do not need to assume that $X$ is bounded. 

\begin{Prop} \label{Punit} Let $(X,d)$ be a bounded invariant metric group. Then, we have that 

$(1)$ $(\mathcal{U}(Lip^1_0(X)),d_{\infty})=(\mathcal{U}(Lip^1_+(X)),d_{\infty})\cong (\overline{X}, d),$
 
$(2)$ $(\mathcal{U}(Lip^1(X)),\theta_{\infty})\cong (\overline{X}\times \R, d+ |.|).$

$(3)$ The group $\mathcal{U}(Lip^1(X))$ is the maximal subgroup of the semigroup $Lip(X)$, having $\delta_e$ as identity element.

\end{Prop}
\begin{proof}
$(1)$ The fact that $(\mathcal{U}(Lip^1_+(X)),d_{\infty})\cong (\overline{X}, d)$, is given in [Corollary 1., \cite{Ba1}]. On the other hand, since, $\mathcal{G}(X)\subset \mathcal{U}(Lip^1_0(X))\subset \mathcal{U}(Lip^1_+(X))$ and since 
$\overline{\mathcal{G}(X)}\cong \overline{X}$ (see [Lemma 2., \cite{Ba1}]) we get that $\mathcal{U}(Lip^1_0(X))=\mathcal{U}(Lip^1_+(X))$. Let us prove the part $(2)$. Indeed, since $\tau$ (Proposition \ref{Ptau}) is an isometric isomorphism, it sends isometrically the group of units onto the group of units. Hence, from Proposition \ref{Ptau} we have $$(\mathcal{U}(Lip^1(X)),\oplus,\theta_{\infty})\cong (\mathcal{U}(Lip^1_0(X)\times \R),\overline{\oplus},d_{\infty}+|.|).$$ Since $\mathcal{U}(Lip^1_0(X)\times \R)=\mathcal{U}(Lip^1_0(X))\times \R$, the conclusion follows from the first part. For the part $(3)$, let $f$ be an element of the maximal group having $\delta_e$ as identity element. Then, $f\oplus \delta_e=f$ and so it follows that $f$ is $1$-lipschitz map i.e $f\in Lip^1(X)$. Thus, $f\in \mathcal{U}(Lip^1(X))$.
\end{proof}

\section{The algebraic case.} \label{S4}
Let $G$ be an algebraic group having $e$ as identity element and let $f : G\longrightarrow \R^+$ be a function, we denote $Osc(f):=\sup_{t,t'\in G}|f(t)-f(t')|$ and by $G^*$ we denote the following set :
$$G^*:=\lbrace f : G \longrightarrow \R^+/ Osc(f)\leq 1\rbrace.$$
Note that the set $G^*$ is juste the set $Lip^1_+(G)$ where $(G,disc)$ is equipped with the discrete metric $"disc"$, which is an invariant complete metric. So, $(G^*, \oplus)$ is a monoid having $\delta_e$ as identity element, where $\delta_e(\cdot):=disc(\cdot,e)$ i.e. $\delta_e(e)=0$ and $\delta_e(t)=1$ for all $t\neq e$. Observe also that two algebraic groups $G$ and $G'$ are isomorphic if and only they are isometrically isomorphic when equipped respectively with the discrete metric. Thus, we obtain that the algebraic group structure of any group $G$ is completely determined by the algebraic monoid structure of $(G^*, \oplus)$.

\begin{Cor} \label{cor1} Let $G$ and $G'$ be two groups. Then the following assertions are equivalent.

$(1)$ the groups $G$ and $G'$ are isomorphic

$(2)$ the monoids $(G^*, \oplus, \rho)$ and $({G'}^*, \oplus, \rho)$ are isometrically isomorphic

$(3)$ the monoids $(G^*, \oplus, d_{\infty})$ and $({G'}^*, \oplus, d_{\infty})$ are isometrically isomorphic (where $d_{\infty}(f,g):=\sup_{t\in G}|f(t)-g(t)|< +\infty$, for all $f, g \in G^*$)

$(4)$ the monoids $(G^*, \oplus)$ and $({G'}^*, \oplus)$ are isomorphic.

\vskip5mm

Moreover, $\Phi : (G^*, \oplus, \rho) \longrightarrow ({G'}^*, \oplus, \rho)$ (resp. $\Phi : (G^*, \oplus, d_{\infty}) \longrightarrow ({G'}^*, \oplus, d_{\infty})$) is an isometric isomorphism of monoids, if and only if there exists an isomorphism of groups $T: G\longrightarrow G'$ such that $\Phi(f)=f\circ T^{-1}$ for all $f\in G^*$.
\end{Cor}
\begin{proof} Since $G^*=Lip^1_+(G)$, where $G$ is equipped with the discrete metric and since $G$ and $G'$ are isomorphic if and only if $(G,disc)$ and $(G',disc)$ are isometrically isomorphic, then the part $(1)\Longleftrightarrow (2)$ is a direct consequence of Theorem \ref{Thm0}. The part $(2)\Longrightarrow (3)$, follows from the fact that $\rho =\frac{d_{\infty}}{1+d_{\infty}}$ by using [Lemma 1, \cite{Ba1}]. The part $(3)\Longrightarrow (4)$ is trivial. Let us prove $(4)\Longrightarrow (1)$. Since an isomorphism of monoids sends the group of unit onto the group of unit, and since the group of unit of $G^*$ (resp. of ${G'}^*)$ is isomorphic to $G$ (resp. to $G'$) by Proposition \ref{Punit}, we get that $G$ and $G'$ are isomorphic. The last assertion is given by Theorem \ref{Thm0}. 

\end{proof}
As mentioned in Remark \ref{rem1}, if $T: G\longrightarrow G'$ is an isomorphism, then $\Phi(f):=f\circ T^{-1} + \inf_G(f)$ for all $f\in G^*$ gives an isomorphism of monoids between $G^*$ and ${G'}^*$ which is not isometric.
\vskip5mm
In the following exemple, we treat the case where $G$ is a finite group.
\begin{Exemp} \label{Exemp1}  Let $n\geq 1$ and $(\R^n,d_{\infty})$ the usual $n$-dimentional space equiped with the max-distance. The subsets $M^n_+$ and $M^n$ of $\R^n$ are defined as follows $$M^n_+:=\lbrace (x_k)_{1\leq k \leq n} \in \R_+^n/ |x_i-x_j|\leq 1, \hspace{2mm} \hspace{2mm} 1\leq i, j \leq n\rbrace.$$
$$M^n:=\lbrace (x_k)_{1\leq k \leq n} \in \R^n/ |x_i-x_j|\leq 1, \hspace{2mm} \hspace{2mm} 1\leq i, j \leq n\rbrace.$$

Let $G:=\lbrace g_1, g_2,... g_n\rbrace$, be a group of cardinal $n$, where $g_1$ is the identity of $G$. We define the law $\star_G$ on $\R^n$ depending on $G$ as follows: for all $x=(x_k)_k, y=(y_k)_k\in \R^n$, 
$$x\star_G y=(z_k)_{1\leq k \leq n},$$ where for each $1\leq k \leq n$,
$$z_k:=\min \lbrace x_i+y_j/  g_i\cdot g_j=g_k, 1\leq i, j \leq n\rbrace.$$ Then, 

$(1)$ The set $(\R^n,\star_G)$ has a semigroup structure (and is abelian if $G$ is abelian). 

$(2)$ The sets $(M^n_+,\star_G)$ and $(M^n,\star_G)$ are monoids having $e=(0,1,1,1,...,1)$ as identity element.

$(3)$ Let $G$ and $G'$ be two groups of cardinal $n$. The monoids $(M^n_+,\star_G)$ and $(M^n_+,\star_{G'})$ are isomorphic if and only if the groups $G$ and $G'$ are isomorphic.

$(4)$ We have that $$\mathcal{U}(M^n_+)\simeq G,$$
$$\mathcal{U}(M^n) \simeq G\times \R.$$
Moreover, the maximal subgroup of $(\R^n,\star_G)$  having  $e$ as identity element is isomorphic to the group  $G\times \R.$

$(5)$ We have that $$Is_m(M^n_+) \simeq Aut(G).$$ 

\end{Exemp}
The properties $(1)-(5)$ follows easily from the results of this note. 
It sufficies to see that the space $\R^n$ can be identified to the space $Lip(G)$ of all real-valued Lipschitz map on $(G,disc)$. Indeed, the map

\begin{eqnarray*}
i: Lip(G)&\longrightarrow& \R^n\\
f &\mapsto& (f(g_1),...,f(g_n))
\end{eqnarray*} 
is a bijective map. Then, we observe that the operation $\star_G$ on $\R^n$ is just the operation $\oplus$ on $Lip(G)$. On the other hand, the subset $M^n_+$ is identified to $Lip^1_+(G)$ and $M^n$ is identified to $Lip^1(G)$. 
\bibliographystyle{amsplain}

\end{document}